\documentclass[11pt]{article}
\usepackage[english]{babel}
\usepackage{amsmath}
\usepackage{amsthm}
\usepackage{amssymb}
\usepackage{mathrsfs}
\usepackage[dvips]{color}
\newtheorem{theorem}{Theorem}[section]
\newtheorem{definition}[theorem]{Definition}
\newtheorem{lemma}[theorem]{Lemma}
\newtheorem{proposition}[theorem]{Proposition}
\newtheorem{corol}[theorem]{Corollary}
\newtheorem{remark}[theorem]{Remark}
\newtheorem{example}[theorem]{Example}

\newcommand{\si}{\mathbb{S}}

\newcommand{\Er}{\mathscr{M}}

\newcommand{\rr}{\mathbb{R}}
\newcommand{\cc}{\mathbb{C}}
\newcommand{\hh}{\mathbb{H}}

\newcommand{\pp}{\partial}

\newcommand{\unx}{\underline{x}}
\newcommand{\uny}{\underline{y}}
\newcommand{\vx}{\vec x\, }

\newcommand{\vp}{\vec p\, }
\newcommand{\vs}{\vec s\, }

\title{\bf Slice monogenic functions }
\author{Fabrizio Colombo\\Dipartimento di Matematica\\Politecnico di
Milano\\Via Bonardi, 9\\20133 Milano,
Italy\\fabrizio.colombo@polimi.it
\\
\and Irene Sabadini\\Dipartimento di Matematica\\Politecnico di
Milano\\Via Bonardi, 9\\20133 Milano,
Italy\\irene.sabadini@polimi.it \and Daniele C.
Struppa\\Department of Mathematics \\and Computer Sciences
\\Chapman University\\Orange, CA 92866 USA,
\\struppa@chapman.edu}

\date{ }
\begin{document}
\maketitle
\begin{abstract}
In this paper we offer a new definition of monogenicity for functions defined
on $\rr^{n+1}$ with values in the Clifford
algebra $\rr_n$ following an idea inspired by the recent papers \cite{gs}, \cite{advances}.
This new class of monogenic functions contains the polynomials (and, more in general,
 power series) with coefficients in the
Clifford algebra $\rr_n$.  We will prove a Cauchy integral
formula as well as some of its consequences. Finally, we deal with the zeroes of some polynomials and
power series.
\end{abstract}


\section{Introduction}
In the past thirty years monogenic functions with values in a Clifford algebra $\rr_n$
have been successfully and intensively studied.
The literature is very rich of results and the studies on the topic are ongoing.
However, one disappointment about monogenic functions is that the identity function
or the powers of the variable considered are not monogenic functions. In the specific case of
the Clifford algebra over two imaginary units, the algebra of quaternions $\hh$, Cullen introduced
a notion whose purpose was exactly to overcome this problem.
Cullen's study (see \cite{cullen}) is based on
the so called intrinsic functions  introduced  by Rinehart for functions
with values in
an algebra. On the basis of
Rinehart's work, Cullen defined a new class of regular functions and it can be shown that the class of
quaternionic regular functions defined by Cullen contains all the
power series of the form $\sum_n q^n a_n$ with real coefficients $a_n$. Recently,
Gentili and Struppa adopted a definition of regularity (slice regularity), see \cite{gs},
\cite{advances}, where they prove that slice regular functions can be
expanded into power series $\sum_n q^n a_n$, with coefficients $a_n\in\hh$.

In this paper, widely inspired by \cite{advances}, we further generalize the ideas therein
to the case of functions
defined on domains of the Euclidean space $\rr^{n+1}$ having values in a Clifford algebra.
We will say that a function $f:\ U\subseteq\rr^{n+1}\to\rr_n$ is slice monogenic if for any
unit  1-vector $I$ the restriction $f_I$ of the function $f$ to the complex plane $x+Iy$
is holomorphic.
We show how slice monogenic functions can be related to power series, we will prove a Cauchy integral
formula as well as some of its consequences. Finally, we deal with the zeroes of polynomials and
power series in the variable $\vx\in\rr^{n+1}$.

An important application of our ideas is a new  way to define a functional calculus for an $n$-tuple
of non commuting operators. In the last section of this paper we provide the algebraic foundations for
such calculus that has been developed in \cite{newfunc}.

Note that in \cite{gs1} the authors used similar ideas to study
the case of functions $f: \rr_3\to\rr_3$: their study shows some interesting geometric peculiarities
of $\rr_3$ which leads the theory in a different direction from the one pursued in this paper.

{\it Acknowledgements} The first and third authors are grateful to
Chapman University for the hospitality during the period in which
this paper was written. They are also indebted to G.N.S.A.G.A. of
INdAM and the Politecnico di Milano for partially supporting their
visit.
\section{Slice monogenic functions}

Let $\rr_n$ be the real Clifford algebra over $n$ units $e_1,\dots
,e_n$ such that $e_ie_j+e_je_i=-2\delta_{ij}$ (see e.g.
\cite{bds}, \cite{csss} or \cite{dss} for the basic notation). An
element in the Clifford algebra will be denoted by $\sum_A e_Ax_A$
where $A=i_1\ldots i_r$, $i_\ell\in \{1,2,\ldots, n\}$,
$i_1<\ldots <i_r$ is a multi-index and $e_A=e_{i_1} e_{i_2}\ldots
e_{i_r}$;  the number $r$ of units in $e_A$ is denoted by $|A|$.
Elements in the linear space $\rr_n^k$ generated by basis vectors
$e_A$ with $|A|=k$ are called $k$-vectors. An element
$\unx\in\rr^n$ can be identified with a 1-vector in the Clifford
algebra: $(x_1,x_2,\ldots,x_n)\mapsto \unx= x_1e_1+\ldots+x_ne_n$
while real numbers will be identified with the 0-vectors, i.e.
with elements in  $\rr_n^0$. A function $f:\ U\subseteq
\rr^n\to\rr_n$ is seen as a function $f(\unx)$ of $\unx$. There
are in the literature several ways to define a notion of
generalized holomorphy for function with values in $\rr_n$. The
most successful is the so called monogenicity which has been
intensively studied during the past thirty years (see for example
\cite{bds}, \cite{csss}, \cite{dss}). A differentiable function
$f$ defined on an open set $U\subseteq\rr^n$ is said to be
monogenic if it is in the kernel of the Dirac operator
$$
\partial_{\unx}=\sum_{i=1}^ne_i\pp_{x_i}.
$$
An important variation of the Dirac operator is the Weyl operator
$$
\partial_{\vx}=\partial_{x_0}+\sum_{i=1}^ne_i\pp_{x_i}
$$
acting on functions $f:U\subseteq\rr^{n+1}\to\rr_n$ and where the variable
in $\rr^{n+1}$ is identified with $\vx=x_0+\unx$. The theory of the
functions in the kernel of the Dirac or of the Weyl operator are equivalent
and the word monogenic is used for functions in the kernel of either of them.
Despite the fact that several results on the holomorphic functions in
one complex variable can be generalized to this setting, the theory of monogenic functions
shows an unpleasant feature: neither the identity function $f(\unx)=\unx$ (or  $f(\vx)=\vx$)
or the powers of the variable
$g(\unx)=\unx^n$ (or $g(\vx)=\vx^n$)  are monogenic.
To introduce our variation of "hyperholomorphic" functions, i.e. the theory of slice monogenic
functions, we will denote
by $\mathbb{S}$ the sphere of unit 1-vectors, i.e.
$$
\mathbb{S}=\{ \unx=e_1x_1+\ldots +e_nx_n\ |\  x_1^2+\ldots +x_n^2=1\}.
$$
The 2-plane $\mathbb{R}+I\mathbb{R}$ passing through $1$ and $I$ will be denoted by $L_I$:
it is a real subspace of $\rr^{n+1}$ isomorphic to the complex plane.
An element $\vx\in L_I$ will be denoted by $x+Iy$.

\begin{definition} Let $U\subseteq\rr^{n+1}$ be a domain and let
$f:\ U\to\rr_n$ be a real differentiable function. Let
$I\in\mathbb{S}$ and let $f_I$ the restriction of $f$ to the
complex line $L_I$.
We say that $f$ is a slice monogenic function (in short s-monogenic function) if for every
$I\in\mathbb{S}$
$$
\frac{1}{2}\left(\frac{\partial }{\partial x}+I\frac{\partial
}{\partial y}\right)f_I(x+Iy)=0.
$$
\end{definition}
However, when dealing with a notion of monogenicity (compare
for example \cite{bds}) there are two possibilities for the
position of the imaginary units so, in this case, it is possible
to introduce an absolutely analogous notion of slice monogenicity on the right.
The theory of right slice monogenic functions is equivalent to the theory of slice
left monogenic functions. In the sequel, we will consider monogenicity on the left and,
for simplicity, we will denote by
$\overline{\partial}_I$ the operator
$\frac{1}{2}\left(\frac{\partial }{\partial x}+I\frac{\partial
}{\partial y}\right)$ and we will refer to the left slice monogenic
function as s-monogenic functions. We will also introduce a notion of $I$-derivative
by means of the operator:
$$
{\partial}_I:=\frac{1}{2}\left(\frac{\partial }{\partial x}-I\frac{\partial
}{\partial y}\right).
$$
\begin{remark}{\rm The s-monogenic functions on $U\subseteq\rr^{n+1}$ form a
right module $\Er (U)$.
In fact it is trivial that if $f,g\in \Er (U)$ then for every
$I\in\si$ one has $\overline{\partial}_If_I=\overline{\partial}_Ig_I=0$,
thus $\overline{\partial}_I(f+g)_I=0$.
Moreover, for any $a\in\rr_n$ we have $\overline{\partial}_I(f_Ia)=
(\overline{\partial}_If)a=0$. It is not true, in general, that the product
of two s-monogenic functions is s-monogenic.
}
\end{remark}
\begin{remark}{\rm As noted in \cite{advances} monomials $\vx^na_n$, $a_n\in\rr_n$
are left s-monogenic, thus also polynomials $\sum_{n=0}^N \vx^na_n$
are s-monogenic. Moreover, any power series $\sum_{n=0}^{+\infty} \vx^na_n$ is left s-monogenic in
its domain of convergence.}
\end{remark}
\begin{definition}
Let $U$ be a domain in $\rr^{n+1}$ and let $f:\ U\to\rr_n$ be an s-monogenic function.
Its s-derivative $\pp_s$ is defined as
\begin{equation}\label{s-derivative}
\pp_s(f)=\left\{\begin{array}{ll}
\pp_I(f)(\vx)\quad &\vx=x+Iy,\ y\not=0\\
\pp_xf (x)\quad &{x\in\rr}.
\end{array}\right.
\end{equation}
\end{definition}

Note that the definition of s-derivative is well posed because it is applied
only to s-monogenic functions. In fact, if a function $f$ is s-monogenic, its
s-derivative in the point $\vec x$ equals $\frac{\pp}{\pp x}f
(\vec x)$ while there can be problems if $f$ is not s-monogenic.
As pointed out in \cite{advances} this phenomenon is peculiar
of the hypercomplex case, as the unit sphere of imaginary numbers
has positive dimension, and does not appear in the complex case since the unit sphere
is only made of
two points.
The following results can be proved as in \cite{advances}.
\begin{proposition}\label{propreg}
1. (\cite{advances}[Proposition 2.6])
The s-derivative $\pp_s(f)$ of an s-monogenic function $f$ is
an s-monogenic function. Moreover $\pp_s^nf(x+Iy)=\displaystyle\frac{\pp^n}{\pp x^n}f(x+Iy)$.
\par\noindent
2.
The s-derivative of a power series $\sum_{n=0}^{+\infty} \vx^na_n$ equals
$\sum_{n=0}^{+\infty} n\vx^{n-1}a_n$ and has the same radius of convergence of the original series.
\end{proposition}
Now we are in need to select $n-1$ unit vectors in $\mathbb{S}$ in order to have a basis
of $\rr_n$ containing the chosen 1-vector $I$.
To this purpose we recall that if $\underline{a}$, $\underline{b}$ are two 1-vectors, then
\begin{equation}\label{provect}
\underline{a}\underline{b}=\langle\underline{a}, \underline{b}\rangle+\underline{a}\wedge \underline{b},
\end{equation}
where $\langle\underline{a}, \underline{b}\rangle$ denotes the scalar product
of $\underline{a}$ and $\underline{b}$ and the wedge product is defined by
$\underline{a}\wedge \underline{b}=\frac{1}{2}(\underline{a} \underline{b}- \underline{b}\underline{a})$.
\begin{proposition}
Let $I=I_1\in\mathbb{S}$. It is possible to choose $I_2,\ldots ,I_n\in\mathbb{S}$ such that
$I_1,\ldots ,I_n$ form a basis for the Clifford algebra $\rr_n$ satisfying the defining relations
$I_rI_s+I_sI_r=-2\delta_{rs}$.
\end{proposition}
\begin{proof}
First of all, note that since $\unx\wedge \uny=-\uny\wedge \unx$, formula (\ref{provect}) gives
$\unx\uny+\uny\unx=2\langle\unx, \uny\rangle$.
Then it sufficient to select the vectors $I_r$ such that $\langle I_s, I_r\rangle=0$ and
$\langle I_r, I_r\rangle=-1$, for $s=1,\ldots,n$, $r=2,\ldots, n$, $s\not= r$. Since $I_r=\sum_{\ell=1}^n x_{r\ell}e_\ell$
we have $\langle I_r, I_r\rangle=-(\sum_{\ell=1}^n x_{r\ell}^2)$ and $\langle I_s, I_r\rangle=
-\sum_{\ell=1}^n x_{r\ell}x_{s\ell}$. By identifying each 1-vector $I_r$ with its components
$(x_1,\ldots ,x_n)\in\rr^n$ we conclude by the Gram Schmidt algorithm.
\end{proof}
\begin{lemma} (Splitting Lemma)
Let $U\subseteq\rr^{n+1}$ be a domain and let $f:\ U\to\rr_n$ be an s-monogenic function. For every choice of $I=I_1\in\mathbb{S}$
let $I_2,\ldots, I_n$ be a completion to an orthonormal basis of $\rr_n$.
Then there exists $2^{n-1}$ holomorphic functions
$F_A\ : U\cap L_I\to L_I$ such that for every $z=x+Iy$
$$
f_I(z)=\sum_{|A|=0}^{n-1} F_A(z) I_A,\quad I_A=I_{i_1}\ldots I_{i_s},
$$
where $A=i_1\ldots i_s$ is a subset of $\{2,\ldots ,n\}$, with ${i_1}<\ldots <{i_s}$,
or, when $|A|=0$, $I_\emptyset=1$.
\end{lemma}
\begin{example}{\rm To make clear the notation of the Lemma, we show an
example before to prove the statement. Let us consider the case of $\rr_4$-valued
functions. A function $f$ can be written as
$$
f=f_0+f_1I_1+f_2I_2+f_3I_3+f_4I_4+f_{12}I_{12}+f_{13}I_{13}+f_{14}I_{14}+f_{23}I_{23}
+
$$
$$
f_{24}I_{24}+f_{34}I_{34}+f_{123}I_{123}+f_{124}I_{124}+f_{134}I_{134}+f_{234}I_{234}
+f_{1234}I_{1234}
$$
and grouping as prescribed in the statement of the Lemma, we obtain
$$
f=(f_0+f_1I_1)+(f_2+f_{12}I_1)I_{2}+(f_3+f_{13}I_{1})I_3+(f_4+f_{14}I_{1})I_4
+
$$
$$
+(f_{23}+f_{123}I_{1})I_{23}+(f_{24}+f_{124}I_{1})I_{24}+(f_{34}+f_{134}I_{1})I_{34}+(f_{234}
+f_{1234}I_{1})I_{234}.
$$
}
\end{example}
\begin{proof} The proof closely follows the proof of the analogue result in \cite{advances}. Given a
function $f=\sum f_A I_A$ let us rewrite it by grouping its components as
$$\sum_{|A|=0}^{n-1} (f_A+f_{1A}I_1)I_A,$$
with obvious meaning of the subscript $1A$. Since $f$ is s-monogenic we have
$\left(\displaystyle\frac{\pp}{\pp x}+I_1\displaystyle\frac{\pp}{\pp y}\right)f_{I_1}(x+I_1y)=0$
and so
$$
\sum \left(\displaystyle\frac{\pp}{\pp x}+I_1\displaystyle\frac{\pp}{\pp y}\right)
(f_A+f_{1A}I_1)I_A=\left(\displaystyle\frac{\pp}{\pp x}f_A+I_1\displaystyle\frac{\pp}{\pp y}f_A
+\displaystyle\frac{\pp}{\pp x}f_{1A}I_1 - \displaystyle\frac{\pp}{\pp y} f_{1A}
\right)I_A=0.
$$
Using the fact that the imaginary units commute with the real valued functions,
we obtain:
$$
\left\{%
\begin{array}{ll}
    \displaystyle\frac{\pp}{\pp x}f_A - \displaystyle\frac{\pp}{\pp y} f_{1A}=0 & \hbox{ } \\
   \displaystyle\frac{\pp}{\pp y}f_A
+\displaystyle\frac{\pp}{\pp x}f_{1A}=0 & \hbox{ } \\
\end{array}%
\right.
$$
for all multi-indices $A$, thus all the functions $F_A=(f_A+f_{1A}I_1)$ satisfy the
standard Cauchy-Riemann
system and therefore they are holomorphic.
\end{proof}
\begin{proposition}\label{coeffTaylor}
If $B=B(0,R)\subseteq\rr^{n+1}$ is a ball centered in $0$ with radius $R>0$, then
$f:\ B\to\rr_n$
is an s-monogenic function if and only if  $f$ has a series expansion of the form
\begin{equation}\label{serie}
f(\vx)=\sum_{m\geq 0}\vx^m\frac{1}{m!}\frac{\pp^mf}{\pp x^m}(0)
\end{equation}
converging on $B$.
\end{proposition}
\begin{proof}
If a function admits a series expansion as in  (\ref{serie}) it is obviously s-monogenic where the series
converges. The converse needs the use of the Splitting lemma and mimics
the proof of Theorem 2.7 in \cite{advances}. Let us consider an element $I=I_1\in\mathbb{S}$ and the corresponding
plane $L_I$. Let $\Delta\subset L_I$ be a disc with center in the origin and radius $r<R$.
The function $f_I$ restriction of $f$ to the plane $L_I$ can be split as $f_I(z)=\sum F_A(z) I_A$,
$z=x+Iy$.
Since every function $F_A$ has values in $L_I$ and is holomorphic,
for any $z\in\Delta$ it admits an integral
representation via the Cauchy formula, i.e.
$$
F_A(z)=\frac{1}{2\pi I}\int_{\pp\Delta(0,r)} \displaystyle\frac{F_A(\zeta)}{\zeta -z}
\, d\zeta,
$$
thus
$$
f_I(z)= \sum_{|A|=0}^{n-1} \left(\frac{1}{2\pi I}\int_{\pp\Delta(0,r)}
\displaystyle\frac{F_A(\zeta)}{\zeta -z}\, d\zeta \right)I_A.
$$
Now observe that as  $\zeta$, $z$ commute being on the same plane $L_I$,
we can expand the denominator in each integral in power series, as in the classical case:
$$
F_A(z)=\frac{1}{2\pi I}\int_{\pp\Delta(0,r)} \sum_{m\geq 0}
\left(\displaystyle\frac{z}{\zeta}\right)^m\displaystyle\frac{F_A(\zeta)}{\zeta}\, d\zeta
$$
$$
=\sum_{m\geq 0}z^m \int_{\pp\Delta(0,r)}
\displaystyle\frac{F_A(\zeta)}{\zeta^{m+1}}\, d\zeta,=\sum_{m\geq 0}z^m\frac{1}{m!}
\frac{\pp^m F_A}{\pp z^m}(0).
$$
Plugging this expression into $f_I(z)=\sum F_A I_A$ we obtain:
$$
f_I(z)=\sum_{|A|=0}^{n-1} \sum_{m\geq 0}z^m\frac{1}{m!}
\frac{\pp^m F_A}{\pp z^m}(0) I_A=\sum_{|A|=0}^{n-1} \sum_{m\geq 0}z^m\frac{1}{m!}
\frac{\pp^m f}{\pp z^m}(0)
$$
and using the definition of s-derivative together with Proposition \ref{propreg}.1,
we get
$$
 \sum_{|A|=0}^{n-1} \sum_{m\geq 0}z^m\frac{1}{m!}
\frac{1}{2}\left(\frac{\pp}{\pp x}-I\frac{\pp}{\pp y}\right)^mf(0)
= \sum_{|A|=0}^{n-1} \sum_{m\geq 0}z^m\frac{1}{m!}
\frac{\pp^m}{\pp x^m}f(0).
$$
Finally observe that the coefficients of the power series do not depend on the choice of
the unit $I$, thus the statement holds for any $I\in\mathbb{S}$.
\end{proof}
\begin{remark}{\rm
The interesting part of Proposition \ref{coeffTaylor} is that, even though the definition of
s-monogenic function depends on the direction of the unit vector $I$,
the coefficients of the series expansion do not depend at all from a choice of $I$.}
\end{remark}
\begin{remark}{\rm
Note that the complex plane $\mathbb{C}$ can be seen both as $\rr^2$ and as the Clifford algebra $\rr_1$. It is immediate to note
that the space of holomorphic functions $f:\ \mathbb{C}\to \mathbb{C}$ coincides with the space
of s-monogenic functions from $g:\ \rr^2\to \rr_1$. For this reason in this paper we will consider the case
of $n>1$ (obviously, all the results we will prove are valid in the case $n=1$, but for holomorphic functions
in one complex variable usually stronger statements hold).}
\end{remark}
Note also that it is possible to interpret holomorphic functions of one complex variable
as a (proper) subset
of the space of s-monogenic functions as shown in the following proposition.
\begin{proposition} Let $I\in\mathbb{S}$ and let us identify $L_I$ with $\mathbb{C}$.
Any holomorphic function $f:\ \Delta (0,R)\subseteq\mathbb{C}\to \cc$ can be extended
(uniquely, up to a choice of an order
for the elements in the basis of $\rr_n$) to an s-monogenic function
$\tilde f: B(0,R)\to\rr_n$.
\end{proposition}
\begin{proof} The function $f(z)$ can be expanded in power series as $f(z)=f(x+Iy)=
\sum_{n=0}^{+\infty} (x+Iy)^na_n$, $a_n\in\mathbb{C}\cong L_I$. Suppose to embed $\mathbb{C}$
into the Clifford algebra $\rr_n$ by identifying the imaginary unit $I\in\mathbb{C}$ with
$I_1\in\rr_n$. Then we define $\tilde f(\vx)=\sum_{n=0}^{+\infty} \vx^n a_n$ which is, obviously,
an s-monogenic function.
\end{proof}
\begin{proposition}\label{composition}
\par\noindent 1. The product of two functions $f,g: \ B(0,R)\to\rr_n$ whose series expansions
have real coefficients is an s-monogenic function.
\par\noindent 2. The composition of an s-monogenic function $f: B(0,R)\to\rr_n$
with an s-monogenic function
$g:\ B(0,R')\to\rr_n$ whose series expansion has real coefficients is an s-monogenic
function where it is defined.
\end{proposition}
\begin{proof}
Suppose that $f(\vx)=\sum_{m\geq 0}\vx^m a_m$ and $g(\vx)=\sum_{r\geq 0}\vx^r b_r$, with $a_m,b_r\in\rr$.
Then $(fg)(\vx)=\sum_{s\geq 0} \vx^s(a_0b_s+a_1b_{s-1}+\ldots+a_sb_0)$ since the coefficients commute with
the variable $\vx$. Now consider $f(g(\vx))$: we have $f(g(\vx))=\sum_{m\geq 0}(\sum_{r\geq 0}\vx^rb_r)^ma_m$.
Since the coefficients $b_r$ commute with the variables we can group them on the right and the statement
follows.
\end{proof}
\begin{corol}
Let $f:\ U\to\rr_n$ be an s-monogenic function and $y_0\in\rr$. Then $f(\vx -y_0)$  is
an s-monogenic function in $U'=\{\vx'=\vx -y_0,\ \vx\in U\}$.
\end{corol}
\begin{proposition}\label{centerreal}
If $B=B(y_0,R)\subseteq\rr^{n+1}$ is a ball centered in
$y_0\in\rr$ with radius $R>0$, then
$f:\ B\to\rr_n$
is an s-monogenic function if and only if  $f$ has a series expansion of the form
\begin{equation}\label{serietraslata}
f(\vx)=\sum_{m\geq 0}(\vx-y_0)^m\frac{1}{m!}\frac{\pp^mf}{\pp x^m}(y_0)
\end{equation}
\end{proposition}
\begin{proof} Consider the transformation of coordinates $\vec z=\vx-y_0$. As the
function $f(\vec z)$ is s-monogenic in a ball centered in the origin with radius $R>0$, we can apply
Proposition \ref{coeffTaylor}. Using the inverse transformation $\vx=\vec z +y_0$, we
obtain the statement.
\end{proof}
\section{Cauchy integral formula and its consequences}
A main result in the theory of monogenic functions is the analogue of the Cauchy integral formula.
In order to state the result for s-monogenic functions we need some notation.
Given an element $\vx=x_0+\unx\in\rr^{n+1}$ let us set
$$
I_{\vx}=\left\{\begin{array}{l}
\displaystyle\frac{\unx}{|\unx|}\quad{\rm if}\ \unx\not=0\\
{\rm any\ element\ of\ } \mathbb{S}{\rm\ otherwise.}\\
\end{array}
\right.
$$
We have the following (compare with \cite{advances}, Theorem 3.5):
\begin{theorem}\label{Cauchyyyy} Let $B=B(0,R)$ be a ball with center in $0$
and radius $R>0$ and let $f:\ B\to\rr_n$ be an s-monogenic function. If $\vx\in B$ then
$$
f(\vx)=\frac{1}{2\pi } \int_{\pp\Delta_{\vx}(0,r)} (\vec \zeta -\vx)^{-1}\,
d\vec\zeta_{I_{\vx}}{f(\vec\zeta)}
$$
where $\vec\zeta\in L_{I_{\vx}} \bigcap B$, $d\vec\zeta_{I_{\vx}}:=-d\vec\zeta{I_{\vx}}$ and $r>0$ is such that
$$
\overline{\Delta_{\vx}(0,r)}=\{x+I_{\vx} y\ |\ x^2+y^2\leq r^2\}
$$
contains $\vx$ and is contained in $B$.
\end{theorem}

\begin{proof}
The proof is based on the Splitting Lemma.
Consider the integral
$$
\frac{1}{2\pi}\int_{\partial \Delta_{\vx}(0,r)} ({\vec\zeta}-\vx)^{-1}d{\vec\zeta}_{I_{\vx}}
{f({\vec\zeta})}
$$
set $I_{\vx}:=I_1$, complete to a basis $I_1,\ldots ,I_n$ of the Clifford algebra $\rr_n$, and
set $f_{I_{\vx}}=\sum_A F_AI_A$. We have:
$$
\frac{1}{2\pi}\int_{\partial
\Delta_{\vx}(0,r)}({\vec\zeta} - \vx )^{-1}d{\vec\zeta}_{I_{\vx}}  {f_{I_{\vx}}({\vec\zeta})}
$$
$$
=\frac{1}{2\pi}\int_{\partial \Delta_{\vx}(0,r)}
({\vec\zeta} - \vx )^{-1}d{\vec\zeta}_{I_{\vx}} \sum_A F_A({\vec\zeta} )I_A
$$
$$
=
\sum_A\frac{1}{2\pi}\int_{\partial \Delta_{\vx}(0,r)}
({\vec\zeta} - \vx )^{-1}d{\vec\zeta}_{I_{\vx}} F_A({\vec\zeta})I_A =\sum_A  F_A(\vx)I_A = f(\vx ).
$$
\end{proof}
\begin{remark}{\rm
Let $B_1=B(0,R_1)$,  $B_2=B(0,R_2)$ be two balls centered in the origin and with radii $0<R_1<R_2$.
The same argument used in the previous proof shows that if a function $f$ is s-monogenic in
a neighborhood of the
annular domain $B_2\setminus B_1$, then for any $\vx\in B_2\setminus B_1$, it  is
$$
f(\vx)=\frac{1}{2\pi } \int_{\pp B_2\cap L_{I_{\vx}}} (\vec\zeta -\vx)^{-1}\,
d\vec\zeta_{I_{\vx}}{f(\vec\zeta)}-\frac{1}{2\pi } \int_{\pp B_1\cap L_{I_{\vx}}} (\vec\zeta -\vx)^{-1}\,
d\vec\zeta_{I_{\vx}}{f(\vec\zeta)}.
$$}
\end{remark}
\begin{remark}{\rm
The function $\mathscr{I}_{\vec y}(\vx)= (\vec x -\vec y)^{-1}$
corresponding to the Cauchy kernel is
not  s-monogenic on $\rr^{n+1}\setminus\{\vec y\}$, unless
$\vec y=y_0\in\rr$. In particular, the function
$\mathscr{I}(\vx)=\vx^{-1}=\displaystyle\frac{\bar
\vx}{|\vx|^2}$, where $\bar \vx=x_0-\unx$, is s-monogenic in $\rr^{n+1}\setminus \{0\}$. Note also that $\mathscr{I}_{y_0}(\vx)$ can be expanded
in power series as $\mathscr{I}_{y_0}(\vx)=\sum_{n\geq 0} y_0^n\vx^{-n-1}$.
}
\end{remark}
\begin{theorem}(Cauchy formula outside a ball) Let $B=B(0,R)$ and
let $B^c=\rr^{n+1}\setminus \overline{B}$. Let $f:\ B^c\to\rr_n$ be an s-monogenic
function with $\lim_{\vx\to\infty}f(\vx)=a$.
If $\vx\in {B}^c$ then
$$
f(\vx)=a-\frac{1}{2\pi}\int_{\pp\Delta_{\vx}(0,r)} ({\vec\zeta} -\vx)^{-1}\, d{\vec\zeta}_{ I_{\vx}}
{f({\vec\zeta})}
$$
where ${\vec\zeta}\in L_{I_{\vx}} \bigcap B^c$, $d{\vec\zeta}_{I_{\vx}}=-d{\vec\zeta}{I_{\vx}}$,
 $0<R<r<|\vx|$ and the complement of the set
 $\overline{\Delta_{\vx}(0,r)}$
 is contained in ${B}^c$ and contains $\vx$.
\end{theorem}
\begin{proof}
The proof is based on the Splitting Lemma. Let $\vx\in\rr^{n+1}\setminus \overline{B}$ and the
corresponding imaginary unit $\vx$.
Consider $r'>r>R$, and the disks
$\Delta=\Delta_{\vx}(0,r)$, $\Delta'=\Delta_{\vx}(0,r')$ on the plane $L_{I_{\vx}}$ having radius $r$ and $r'$
respectively such that $\Delta'\ni\vx$.
Since $f$ is s-monogenic on $\Delta'\setminus\Delta$ we can apply
the Cauchy formula in
$\overline{\Delta'}\setminus \Delta$ to compute $f(\vx)$. We obtain:
$$
f(\vx)=\frac{1}{2\pi}\int_{\partial \Delta'\setminus\partial \Delta} ({\vec\zeta}-\vx)^{-1}d{\vec\zeta}_{I_{\vx}}
{f({\vec\zeta})}=\frac{1}{2\pi}\int_{\partial \Delta'} ({\vec\zeta}-\vx)^{-1}d{\vec\zeta}_{I_{\vx}}
{f({\vec\zeta})}-\frac{1}{2\pi}\int_{\partial \Delta} ({\vec\zeta}-\vx)^{-1}d{\vec\zeta}_{I_{\vx}}
{f({\vec\zeta})}.
$$
Let us set $I_{\vx}:=I_1$ and complete to a basis $I_1,\ldots ,I_n$ of the Clifford algebra $\rr_n$.
The Splitting Lemma gives $f_{I_{\vx}}=\sum_A F_AI_A$ and we can write:
$$
\frac{1}{2\pi}\int_{\partial \Delta'}
({\vec\zeta} - \vx )^{-1} d{\vec\zeta}_{I_{\vx}} {f({\vec\zeta})}
-\frac{1}{2\pi}\int_{\partial \Delta}
({\vec\zeta} - \vx )^{-1} d{\vec\zeta}_{I_{\vx}} {f({\vec\zeta})}
$$
$$
=
\sum_A\frac{1}{2\pi}\int_{\partial \Delta'}
({\vec\zeta} - \vx )^{-1}d{\vec\zeta}_{I_{\vx}} F_A({\vec\zeta})I_A -
\sum_A\frac{1}{2\pi}\int_{\partial \Delta}
({\vec\zeta} - \vx )^{-1}d{\vec\zeta}_{I_{\vx}} F_A({\vec\zeta})I_A .
$$
Let us now consider a single component $F_A$ at a time. By computing the integral on $\pp\Delta'$
in spherical coordinates, and letting $r'\to\infty$ we obtain that the integral equals
$a_A=\lim_{r'\to\infty}F_A$, therefore:
$$
F_A(\vx)=a_A -
\frac{1}{2\pi}\int_{\partial \Delta}
({\vec\zeta} - \vx )^{-1}d{\vec\zeta}_{I_{\vx}} F_A({\vec\zeta})I_A=a_A-F_A(\vx).
$$
Taking the sum of the various components multiplied with the corresponding
units $I_A$ we get the statement with $a=\sum_A a_AI_A$.
\end{proof}
\begin{theorem} (Cauchy estimates) Let $B=B(0,R)$ be the ball centered in $0$ having radius $R>0$.
Let $f:\ B \ \to \rr_n$ be an s-monogenic function, $I\in\mathbb{S}$ and $0<r<R$.
Set $\pp \Delta_I(0,r)=\{(x+Iy)\ |\ x^2+y^2=r^2 \}$, $M_I=\max \{ |f(\vx)|\ | \vx\in \pp \Delta_I(0,r)\}$
and $M=\inf\{M_I\ |\ I\in\mathbb{S}\}$. Then
$$
\frac{1}{n!}\left|\frac{\pp^n f}{\pp x^n}(0)\right|\leq \frac{M}{r^n}, \quad n\geq 0.
$$
\end{theorem}
\begin{proof}
The result easily follows from the Splitting Lemma and the corresponding proof in the complex case.
\end{proof}
Using the previous result it is immediate to show the following
\begin{theorem}(Liouville) Let $f:\ \rr^{n+1}\to\rr_n$ be an entire s-monogenic function.
If $f$ is bounded then $f$ is constant on $\mathbb{R}^{n+1}$.
\end{theorem}
\begin{proof}
Suppose that $|f|\leq M$ on $\rr^{n+1}$. By the previous theorem we have:
$$
\frac{1}{n!}\left|\frac{\pp^n f}{\pp x^n}(0)\right|\leq \frac{M}{r^n}, \quad n\geq 0,
$$
and letting $r\to+\infty$ we obtain $\frac{\pp^n f}{\pp x^n}(0)=0$ for any $n>0$ and this implies
$f(\vx)=c$, $c\in\rr_n$.
\end{proof}
\begin{corol}
Let $f:\ \rr^{n+1}\to\rr_n$ be an entire s-monogenic function.
If $\lim_{\vx\to \infty}f$ exists then $f$ is constant on $\mathbb{R}^{n+1}$.
\end{corol}
\begin{theorem}\label{int_nullo}
Let $U$ be an open set in $\rr^{n+1}$.
If $f:\ U\to\rr_n$ is s-monogenic then for every $I\in\mathbb{S}$ and for any closed, simple
curve $\gamma_I\subset U\cap L_I$ we have
$$
\int_{\gamma_I} d\vx f(\vx) =0.
$$
\end{theorem}
\begin{proof} It is an easy consequence of the Splitting Lemma and the analogue result for
holomorphic functions.\end{proof}
In principle, the analogue of the main theorems which hold for holomorphic functions can be proved
also in this setting, with minor changes in their proofs. We mention below some of them:
\begin{theorem} (Identity principle).
Let $U$ be an open set in $\rr^{n+1}$ such that $U\cap \rr$ has an accumulation point.
Let $f: U\ \to \rr_n$ be an s-monogenic function,
and $Z$ the set of its zeroes. If there is an imaginary unit $I$
such that $L_I \cap Z$ has an accumulation point, then $f\equiv 0$ on $U$.
\end{theorem}
\begin{proof}
Let us consider the restriction $f_I$ of $f$ to the line $L_I$. By the Splitting Lemma we have
$$
f_I(z)=\sum_{|A|=0}^{n-1} F_A I_A
$$
with $F_A$ holomorphic for every multi-index $A$.
Since $L_I\cap Z$ has an accumulation point,
we deduce that all the functions $F_A$ vanish, thus $f_I=0$. In particular $f_I$ vanishes in
the points of $U$ on the real axis. Any other plane $L_{I'}$ is such that $f_{I'}$ vanishes
on $U\cap\rr$ which has an accumulation point, thus its components $F_{A'}$ vanish on $U\cap \rr$. This fact implies that also
$f_{I'}$ vanish on $L_{I'}$, thus $f\equiv 0$ on $U$.
\end{proof}
\begin{corol}
Let $U$ be an open set in $\rr^{n+1}$ such that $U\cap \rr$ has an accumulation point.
Let $f,g: U\ \to \rr_n$ be s-monogenic functions. If there is an imaginary unit $I$
such that $f= g$ on a subset of $L_I$ having an accumulation point, then $f\equiv g$ on $U$.
\end{corol}
\begin{corol}
Let $B=B(x_0,R)$, $x_0\in\rr$, and $f,g:\ B\ \to \rr_n$ be s-monogenic functions. If there exists
$I\in\mathbb{S}$ such that $f= g$ on
a subset of $L_I\cap B$ having an accumulation point, then $f\equiv g$ on $B$.
\end{corol}
\begin{corol}
Let  $B=B(x_0,R)$, $x_0\in\rr$, and let $f:\ B\ \to \rr_n$ be an s-monogenic function. Then $f\equiv 0$ on $B$ if and only if
$\pp^n_sf(0)=0$ for all $n\in\mathbb{N}$.
\end{corol}
\section{Laurent series}
\begin{proposition}
Let $f:\ B(0,R)\to\rr_n$ be the s-monogenic function expressed by the series
$\sum \vx^m a_m$ converging on $B$.
Then the function $f\cdot\mathscr{I}$ is s-monogenic on $\rr^{n+1}\setminus B(0,1/R)$
and it can be expressed by
the series $\sum \vx^{-m}a_m$ converging on $\rr^{n+1}\setminus B(0,1/R)$.
\end{proposition}
\begin{proof}
Proposition \ref{composition} implies that $f\cdot\mathscr{I}$ is an s-monogenic
function on $\rr^{n+1}\setminus B(0,1/R)$. The statement follows from the analogue result
for holomorphic functions in one complex variable.
\end{proof}
\begin{theorem}
Let $f$ be an s-monogenic function in an annular domain $A=\{\vx \in\rr^{n+1}\ |\ R_1<|\vx |<R_2 \}$,
$0<R_1<R_2$. Then $f$ admits the following unique Laurent expansion
\begin{equation}\label{Laur}
f(\vx )=\sum_{m=0}^{+\infty} \vx^m a_m +\sum_{m=1}^{+\infty} \vx^{-m} b_{m}
\end{equation}
where
$a_m=\frac{1}{m!}\pp_s^mf(0)$ and $b_m = \frac{1}{2\pi } \int_{\pp B(0,R'_1)\cap L_{I_{\vx}}}
{\vec \zeta}^{m-1} d{\vec\zeta}_{I_{\vx}}{f({\vec\zeta})}$.
The two series in (\ref{Laur}) converge in the open ball $B(0,R_2)$ and
$\rr^{n+1}\setminus \overline{B(0,R_1)}$, respectively.
\end{theorem}
\begin{proof}
Let $\vx\in A$, then there exist two positive real numbers $R_1'$, $R'_2$ such that
$A'=\{\vx \in\rr^{n+1}\ |\ R'_1<|\vx |<R'_2 \}\subset A$, and $\vx \in A'$. Using the Cauchy
integral formula, we can write
$$
f(\vx)=\frac{1}{2\pi } \int_{\pp A'\cap L_{I_{\vx}}} ({\vec\zeta} -\vx)^{-1}\,
d{\vec\zeta}_{I_{\vx}}{f({\vec\zeta})}=f_1(\vx )+f_2(\vx)
$$
where
$$
f_1(\vx)=\frac{1}{2\pi } \int_{\pp B(0,R'_2)\cap L_{I_{\vx}}} ({\vec\zeta} -\vx)^{-1}\,
d{\vec\zeta}_{I_{\vx}}{f({\vec\zeta})}
$$
and
$$
f_2(\vx)=-\frac{1}{2\pi } \int_{\pp B(0,R'_1)\cap L_{I_{\vx}}} ({\vec\zeta} -\vx)^{-1}\,
d{\vec\zeta}_{I_{\vx}}{f({\vec\zeta})}.
$$
The first integral is associated to the first series in the Laurent expansion, thanks
to Proposition \ref{coeffTaylor}.
Let us consider the second integral. Using the Splitting Lemma, we can reason
as in the case of functions in one complex variable, and consider the single
 components of $f_2(\vx )$.
In $\rr^{n+1}\setminus \overline{B(0,R_1')}$, we have
$$
F_A(\vx)=-\frac{1}{2\pi } \int_{\pp B(0,R'_1)\cap L_{I_{\vx}}} ({\vec\zeta} -\vx)^{-1}\,
d{\vec\zeta}_{I_{\vx}}{F_A({\vec\zeta})}=\frac{1}{2\pi } \int_{\pp B(0,R'_1)\cap L_{I_{\vx}}} \sum_{m\geq 0}
\vx^{-m-1}{\vec\zeta}^m
d{\vec\zeta}_{I_{\vx}}{F_A({\vec\zeta})}
$$
where we have used the fact that on the plane $L_{I_{\vx}}$ the variables ${\vec\zeta}$ and $\vx$ commute.
Now, using the uniform convergence of the series we can write
$$
F_A(\vx)= \sum_{m\geq 0}\vx^{-m-1} \frac{1}{2\pi } \int_{\pp B(0,R'_1)\cap L_{I_{\vx}}}
{\vec\zeta}^m d{\vec\zeta}_{I_{\vx}}{F_A({\vec\zeta})}=  \sum_{m\geq 0}\vx^{-m-1}  b_{m+1,I_{\vx};A}
$$
where
$$
b_{m+1,I_{\vx};A}:= \frac{1}{2\pi } \int_{\pp B(0,R'_1)\cap L_{I_{\vx}}}
{\vec\zeta}^m d{\vec\zeta}_{I_{\vx}}{F_A({\vec\zeta})}.
$$
We can write:
$$
\tilde f_2(\vx)=\sum_A F_A(\vx)I_A=\sum_{m\geq 0}\sum_A  \vx^{-m-1}  b_{m+1,I_{\vx};A} I_A
$$
however, $\tilde f_2(\vx) $ coincides with $f_2(\vx)$ on the plane $L_{I_{\vx}}$, thus they coincide
everywhere and
 the coefficients $b_{m+1, I_{\vx};A}$ do not depend on the choice of the imaginary unit $I_{\vx}$. The statement follows.
\end{proof}
\begin{remark}
{\rm An analogue result holds for functions s-monogenic in an annular domain of the type $A=\{\vx \in\rr^{n+1}\ |\ R_1<|\vx -y_0 |<R_2, \ y_0\in\rr \}$. In fact, it is sufficient to reason as in the proof of Proposition \ref{centerreal}.}
\end{remark}

\section{Some polynomial equations and series}
As far as we know, in the literature there are no results about the zeroes of monogenic functions
(here we mean in the sense of \cite{bds}) or about the zeroes of polynomials in the variables $\vx$ or $\unx$
(which, of course, are not interesting from the Clifford analysis point of view since they do not correspond
to monogenic functions). As Clifford algebras are not division algebras, the Fundamental Theorem
of Algebra does not hold, thus we cannot guarantee that a given polynomial in the variable $\vx$
has a zero, not even if it is a degree one polynomial. However, if a given polynomial has a zero, it
is interesting to characterize its vanishing set. As we will see in the sequel, the set of zeroes can
contain $(n-1)$--spheres in $\rr^{n+1}$.
To describe those spheres it is useful to introduce the following
notation: let $\vec s =s_0+\sum_{i=1}^ns_i e_i$ and let
$$
[\vec s\,]=\{\vx\in\rr^{n+1}\ | \ x_0=s_0,\ |\vx|=|\vec s\,| \}.
$$
It is immediate to note that the relation $\vx\sim\vec s$ if and only if $x_0=s_0$, $|\vx|=|\vec s\,\,|$
is an equivalence relation. An equivalence class
contains only the element $\vec s$ when it is a real number, while it contains infinitely many
elements when $\vec s$ is not real and corresponds to an $(n-1)$--dimensional sphere in $\rr^{n+1}$.
It is well known (see e.g. \cite{lam}) that
in the skew field of quaternions, the equivalence class of a quaternion $s$ is characterized by a
quadratic equation.
A first, yet interesting, result is that the same fact holds also in a (non division)
Clifford algebra
$\rr_n$ if $\vec s\in\rr^{n+1}$ and one looks for solutions in $\rr_n^0\oplus \rr_n^1$:
\begin{proposition} Let $\vec s=s_0+\sum_{i=1}^ns_i e_i\in\rr^{n+1}$.
Consider the equation
\begin{equation}\label{eqaurea}
\vx^2-2 Re[\vec s]\vx +|\vec s\,|^2=0.
\end{equation}
Then,  $\vx=x_0+\unx$, $x_0\in\rr_n^0=\rr$, $\unx\in\rr_n^1$ is a solution if and only if
$\vx\in [\vec s\,]$.
\end{proposition}
\begin{proof} The result is immediate when $\vec s=s_0\in\rr$.
Let us suppose that $\vec s\not\in\rr$. It is immediate that $\vx\in [\vec s]$ is a solution.
Conversely, let  $\vx$
be a solution, i.e. $(x_0+\unx)^2-2Re[s](x_0+\unx)+|\vec s\,|^2=0$.
A direct computation shows that this is possible if and only if $\unx=0$ or $x_0=s_0$. The first
possibility does not give any solution, while the second gives $|\vx|=|\vec s\,|$, i.e. the equivalence
class of $\vec s$.
\end{proof}
\begin{remark}{\rm The $(n-1)$-sphere corresponding to $[\vec s\,]$ contains elements of the type
$s_0+I|\vec s|$, with $I$ varying in all the possible ways in $\mathbb{S}$.
}
\end{remark}
\begin{remark}{\rm It is not true, in general, that
equation (\ref{eqaurea}) characterizes an equivalence
class if we consider generic vectors in $\rr_n$.
Consider for example $s=(1-e_{123})$. Then $s$ itself is not a solution of the
equation.}
\end{remark}

Following \cite{advances}, Theorem 5.1, we can prove the following
theorem which establish that if the vanishing set of an s-monogenic function contains
two different points on the same $(n-1)$--sphere, then it contains the whole sphere:
\begin{theorem}
Let $\sum_{m\geq 0}\vx^{m}a_{m}$ be the power series with radius of convergence $R$
associated to a given s-monogenic function. If there are two different elements in
a given equivalence class $[\vec s]$, both solutions to the equation
$$
\sum_{m\geq 0}^{+\infty}\vx^{m}a_{m}=0,
$$
then all the elements in the equivalence class are solutions.
\end{theorem}
\begin{proof}
Suppose that there exist $x_{0}+I_1
    y_{0}$, $x_{0}+I_2
    y_{0}\in [\vec s]$ and $I_2\not= I_1 \in \mathbb{S}$,  such
    that
    \begin{equation}\label{eqn}
    \sum_{m\geq 0}^{+\infty}(x_{0}+y_{0}I_1)^{m}a_{m}=0
   \end{equation}
    and
    \begin{equation}\label{eqnn}
    \sum_{m\geq 0}^{+\infty}(x_{0}+y_{0}I_2)^{m}a_{m}=0.
    \end{equation}
For any fixed $m\in \mathbb{N}$ and any $I\in \mathbb{S}$ we have that
    \begin{equation}\label{eqnnn}
    (x_{0}+y_{0}I)^{m} = \sum_{i=0}^{m}{m \choose
    i}x_{0}^{m-i}y_{0}^{i}I^{i}= \alpha_{m}+I\beta_{m}.
    \end{equation}
Equations (\ref{eqn}), (\ref{eqnn}) and equality (\ref{eqnnn}) give, by absolute
convergence,
    $$
    0=\sum_{m\geq 0}(\alpha_m +I_1 \beta_m)a_{m}
    -\sum_{m\geq 0}(\alpha_m + I_2 \beta_m)a_{m}=
    $$
    $$
    \sum_{m\geq 0}
    \left((\alpha_{m}+I_1\beta_{m})-(\alpha_{m}+I_2\beta_{m})\right)a_{m}=
    $$
    $$
    \sum_{m\geq 0}
    \left(I_1\beta_{m}-I_2\beta_{m}\right)a_{m}=
   (I_1-I_2)\left(\sum_{m\geq 0}
    \beta_{m}a_{m}\right).
    $$
Note that $I_1-I_2$ is an invertible element, therefore
    $
    \sum_{m\geq 0}
    \beta_{m}a_{m}=0
    $
    and similarly, by (\ref{eqn}) also $
    \sum_{m\geq 0}
    \alpha_{m}a_{m}=0.
    $
   Using (\ref{eqnnn}) we get that for any $I\in \mathbb{S}$
     $$
    \sum_{m\geq 0}(x_{0}+y_{0}I)^{m}a_{m}=
    \sum_{m\geq 0}\alpha_{m}a_{m}+I\left(\sum_{m\geq 0}^{+\infty}\beta_{m}a_{m}
    \right)=0.
    $$
\end{proof}
\begin{remark}{\rm Whenever a polynomial $p(\vx)$ with complex coefficients
(i.e. coefficients on a plane $L_I$) has two conjugate complex roots, all the points on the sphere
defined by any of two points is a root of the polynomial. }
\end{remark}
\begin{remark}{\rm In a Clifford algebra, the Euclidean algorithm is not allowed, in general. Thus,
if a polynomial in the variable $\vx$ factors as $p(\vx )=(\vx
-\vec a) q(\vx)$ then $\vx=\vec a$ is a solution of $p(\vx )=0$
but knowing a solution $\vx=\vec a$ is
of the equation $p(\vx )=0$ does not allow to divide the
polynomial by$(x-\vec a)$. }
\end{remark}

\subsection{The non commutative Cauchy kernel series }

We conclude this section  with the study of the noncommutative Cauchy kernel series
\begin{equation}\label{NCCKS}
\sum_{n\geq 0}\vp^n \vs^{-1-n},
\end{equation}
for $\vp\vs\not=\vs\vp$,
which is a fundamental tool to define a functional calculus for several noncommuting operators, see \cite{newfunc}.
To motivate the results of this section we make the following observations.
We consider the Cauchy integral formula given by Theorem \ref{Cauchyyyy} and we expand in power series the Cauchy kernel $(\vs -\vp)^{-1}$.
Now, let  $T=T_0+T_1e_1+...+T_ne_n$ be a bounded linear operator where $T_j$, $j=0,...,n$ are
linear bounded operators in the usual sense acting on a Banach space.
\par\noindent
To obtain a functional calculus we have to  formally replace  $\vp$ by $T$ in the non commutative Cauchy kernel series (\ref{NCCKS}) and we have to get the correct resolvent operator for $T$. Thus it is necessary to sum the series $\sum_{n\geq 0}\vp^n \vs^{-1-n}$ in the case $\vp\vs\not=\vs\vp$
and, moreover, we have to observe that this sum does not depend on the commutativity of the real components of $\vp$, see Remark \ref{remimport}.
 Indeed, the components of $T$ in general do not commute.
\par\noindent
The sum
of $\sum_{n\geq 0}\vp^n \vs^{-1-n}$ allows us to define a new notion of resolvent operator  for the $(n+1)$-tuple $T$ of non commuting  operators $T_j$ and, as a consequence, a new eigenvalue equation.
\par\noindent
The most important results of this section is Corollary \ref{corCauchy}, to prove it
we begin with  the following:

\begin{theorem} \label{th5.9}Let $\vec p$,
$\vec s\,\in\rr_n^0\oplus\rr_n^1$ be such that $\vp\vs\not=\vs\vp$.
Then
$
-(\vp-\overline{\vs})^{-1}(\vp^2-2\vp Re [\vs]+|\vs|^2),
$
is the inverse of the noncommutative Cauchy kernel series (\ref{NCCKS})
which is convergent for $|\vp|<|\vs|$.
\end{theorem}

\begin{proof}
Let us verify that
$$
-(\vp-\overline{\vs})^{-1}(\vp^2-2\vp Re [\vs]+|\vs|^2)\ \sum_{n\geq 0} \vp^n \vs^{-1-n}=1.
$$
We therefore obtain
\begin{equation}\label{SS1}
(-|\vs|^2-\vp^2+2\vp Re [\vs])\sum_{n\geq 0} \vp^n \vs^{-1-n}=\vs+\vp-2 \ Re [\vs].
\end{equation}
Observing that $-|\vs|^2-\vp^2+2\vp Re [\vs]$ commutes with $\vp^n$ we can
rewrite this last equation as
$$
\sum_{n\geq 0} \vp^n(-|\vs|^2-\vp^2+2\vp Re [\vs]) \vs^{-1-n}=\vs+\vp-2 \ Re [\vs].
$$
Now the left hand side can be written as
$$
\sum_{n\geq 0} \vp^n(-|\vs|^2-\vp^2+2\vp Re [\vs]) s^{-1-n}=
$$
$$
=-\Big(|\vs|^2\vs^{-1}+\vp(-2\vs Re [\vs] +|\vs|^2)\vs^{-2} +\vp^2(\vs^2-2\vs Re
[\vs]+|\vs|^2)\vs^{-3} +\vp^3(\vs^2-2\vs Re [\vs]+|\vs|^2)\vs^{-4}+...\Big).
$$
Using the identity (\ref{eqaurea}):
$
\vs^2-2\vs Re [\vs]+|\vs|^2=0
$
we get
$$
\sum_{n\geq 0} \vp^n(-|\vs|^2-\vp^2+2\vp Re [\vs])\vs^{-1-n}=
\vs-2 \ Re [\vs]+\vp
$$
which equals the right hand side of (\ref{SS1}).
\end{proof}
\begin{remark}\label{remimport}
{\rm Observe that in the proof of  Theorem \ref{th5.9} we have not used the the fact that
the components of $\vp$ commute.
}
\end{remark}
\par\noindent
A direct consequence of Theorem \ref{th5.9} is that we can explicitly write the sum of the
noncommutative Cauchy kernel series:
\begin{corol}\label{corCauchy}
Let $\vec p$,
$\vec s\,\in\rr_n^0\oplus\rr_n^1$ be such that $\vp\vs\not=\vs\vp$.
Then
$$
\sum_{n\geq 0}\vp^n\vs^{-1-n} =-(\vp^2-2\vp Re [\vs]+|\vs|^2)^{-1}(\vp-\overline{\vs}),
$$
for $|\vp|<|\vs|$.
\end{corol}
\begin{remark}{\rm
We now observe that the expression
$(\vp^2-2\vp Re[\vs]+|\vs|^2)^{-1}(\vp-\overline{\vs})$ involves an inverse
which does not exist if we set $\vp=\overline{\vs}$; indeed, in this
case we have $\overline{\vs}^2-2\overline{\vs}Re[\vs]+|\vs|^2=0$. On may wonder if
the factor $(\vp-\bar \vs)$ can be simplified. However, it can be shown that
this is not possible and the function
$(\vp^2-2\vp Re[\vs]+|\vs|^2)^{-1}(\vp-\overline{\vs})$
 cannot be extended to a continuous function in
$\vp=\overline{\vs}$.}
\end{remark}


\begin{thebibliography}{99}

\bibitem{bds} F. Brackx, R. Delanghe, F. Sommen, {\em Clifford Analysis},
Pitman Res. Notes in Math., 76, 1982.

\bibitem{csss}
 F. Colombo, I. Sabadini, F. Sommen, D.C.
Struppa, {\em Analysis of Dirac Systems and Computational Algebra},
Progress in Mathematical Physics, Vol. 39, {Birkh\"auser}, Boston,
2004.

\bibitem{newfunc}
 F. Colombo, I. Sabadini, Struppa,
{\em A new functional calculus for noncommuting operators},
 arXiv:0708.3594.


\bibitem{cullen} C.G. Cullen, An integral theorem for analytic
intrinsic functions on quaternions, Duke Math. J., {\bf 32} (1965),
139-148.

\bibitem{dss} R. Delanghe, F. Sommen, V.Soucek, {\em Clifford Algebra and
Spinor-valued Functions}, Mathematics and Its Applications 53, Kluwer Academic Publishers, 1992.

\bibitem{gs} G. Gentili, D.C. Struppa, {\em A new approach to Cullen-regular functions
of a quaternionic variable},
 C.R. Acad. Sci. Paris, {\bf 342} (2006), 741--744.


\bibitem{advances}
G. Gentili, D. C. Struppa, {\it A new theory of regular functions of a quaternionic variable},
to appear on {\it Advances in Math.}, 2007.


\bibitem{gs1} G. Gentili, D.C. Struppa, {\em Regular functions on a Clifford Algebra},
 to appear in Compl. Var. Ell. Eq.

\bibitem{lam} T. Y.  Lam, {\em A first course in noncommutative rings},
Second edition, Graduate Texts in Mathematics,
131 Springer-Verlag, New York, 2001.


\end{thebibliography}
\end{document}